\documentclass{article}
\usepackage{latexsym,amsfonts,amsthm,amsmath,amscd,amssymb}
\usepackage[dvips]{graphicx}

\newtheorem{lemma}{Lemma}[section]
\newtheorem{thm}[lemma]{Theorem}
\newtheorem{rem}[lemma]{Remark}
\newtheorem{prop}[lemma]{Proposition}
\newtheorem{cor}[lemma]{Corollary}

\newtheorem{defn}[lemma]{Definition}

\newcommand\matR{{\mathbb{R}}}

\newcommand\calD{{\mathcal D}}

\begin{document}

\title{Diffeological Levi-Civita connections}

\author{Ekaterina~{\textsc Pervova}}

\maketitle

\begin{abstract}
\noindent A diffeological connection on a diffeological vector pseudo-bundle is defined just the usual one on a smooth vector bundle; this is possible to do, because there is a standard diffeological 
counterpart of the cotangent bundle. On the other hand, there is not yet a standard theory of tangent bundles, although there are many suggested and promising versions, such as that of the internal tangent 
bundle, so the abstract notion of a connection on a diffeological vector pseudo-bundle does not automatically provide a counterpart notion for Levi-Civita connections. In this paper we consider the dual of the 
just-mentioned counterpart of the cotangent bundle in place of the tangent bundle (without making any claim about its geometrical meaning). To it, the notions of compatibility with a pseudo-metric and 
symmetricity can be easily extended, and therefore the notion of a Levi-Civita connection makes sense as well. In the case when $\Lambda^1(X)$, the counterpart of the cotangent bundle, is finite-dimensional, 
there is an equivalent Levi-Civita connection on it as well.

\end{abstract}

\section*{Introduction}

Diffeological spaces (introduced originally in \cite{So1}, \cite{So2}; see \cite{iglFibre} for their early development, and \cite{iglesiasBook} for a recent and comprehensive treatment) are a generalization of the
usual smooth manifolds, and the theory describing them includes counterparts of a great number of the standard constructions associated to smooth manifolds. One item noticeably lacking, though, is 
diffeology's own version of the tangent bundle; there is not a standard, universally agreed upon, construction, although many different ones have been proposed, such as the \emph{internal tangent 
bundle} \cite{CWtangent} and the \emph{bundle of $1$-tangent vectors} \cite{iglesiasBook}. This is of course an obstacle for carrying over to the diffeological context whatever notion that
uses tangent vectors, with covariant derivatives being an instance of this.

In \cite{connections}, of which the present work is a sequel, we considered a version of a diffeological connection (which is different from the one briefly considered in \cite{iglesiasBook}) that mimics as much 
as possible the usual notion of a connection on a smooth vector bundle $E$ over a smooth manifold $M$ as an operator $C^{\infty}(M,E)\to C^{\infty}(M,T^*M\otimes E)$. We just have a diffeological vector 
pseudo-bundle $V\to X$ instead of $E\to M$ and the pseudo-bundle $\Lambda^1(X)$ of diffeological differential $1$-forms instead the standard cotangent bundle. This works as an abstract definition well 
enough, and as a substitute of the tangent bundle, we simply used the dual pseudo-bundle $(\Lambda^1(X))^*$, mainly for the reason that it does reflect the usual duality between $TM$ and $T^*M$ (although 
there may not be a diffeomorphism by duality).

In the present work we continue on this subject, defining Levi-Civita connections on $(\Lambda^1(X))^*$ and on $\Lambda^1(X)$. The basic definition is identical to the standard one in substance (a so-called 
\emph{pseudo-metric} is used instead of a Riemannian one). The assumption that $\Lambda^1(X)$ admits a pseudo-metric $g$, so in particular it has finite-dimensional fibres, is used throughout. The abstract 
definitions are almost identical to the standard ones, and the corresponding notion of a diffeological Levi-Civita connection turns out to have, under the appropriate assumptions, the same uniqueness property 
at least. In order to give a more concrete angle to this topic we consider the behavior of the Levi-Civita connections on $\Lambda^1(X)$ under the operation of diffeological gluing; however, we only treat 
the case of gluing along a diffeomorphism (a diffeological one, so this is still much more general than the standard case, since its domain of definition can be anything). Having to put this restriction in 
is somewhat disappointing (the more interesting applications are more likely to be found for non-diffeomorphic bluings); but at least we are able to give a more or less complete treatment, showing in the end 
that the connection on $\Lambda^1(X_1\cup_f X_2)$ induced by the Levi-Civita connections on the factors, that is, on $\Lambda^1(X_1)$ and $\Lambda^1(X_2)$, is itself the Levi-Civita connection on 
$\Lambda^1(X_1\cup_f X_2)$ relative to the appropriate pseudo-metric.

\paragraph{Acknowledgements} This work owes much to an indirect (but no less significant for that) contribution from Prof. Riccardo Zucchi.

\section{Definitions}

We now recall some of the main notions of diffeology. After that, we go into some detail about diffeological differential forms and the abstract notion of a diffeological connection.
The basic concepts have appeared originally in \cite{So1} and \cite{iglFibre} (the \emph{diffeological gluing} is a partial case of quotienting, and the \emph{gluing diffeology} is an instance of the quotient
diffeology); a recent and comprehensive source is \cite{iglesiasBook}.

\subsection{Diffeological spaces, diffeological gluing, and pseudo-bundles}

The notion of a \textbf{diffeological space} \cite{So1} is a generalization of that of a smooth manifold; it is defined as any set $X$ endowed with a \textbf{diffeology} $\calD$. A diffeology (or a
\textbf{diffeological structure}) is a collection of maps $U\to X$, called \textbf{plots}, for all open sets $U$ in all $\matR^n$, that includes all constant maps and all pre-compositions with the
usual smooth maps into $U$, and that satisfies the following \emph{sheaf condition}:
\begin{itemize}
\item for any $p:U\to X$ and for any open cover $\{U_i\}$ of $U$, all restrictions $p|_{U_i}$ being plots implies that $p$ itself is a plot.
\end{itemize}
If $X$ and $Y$ are two diffeollogical spaces, $f:X\to Y$ is said to be \textbf{smooth} if its pre-composition with any plot of $X$ is a plot of $Y$. \emph{Vice versa}, if any plot of the target space $Y$ locally lifts 
to a plot of $X$, the map $f$ is said to be a \textbf{subduction}; the diffeology of $Y$ is said to be the \textbf{pushforward diffeology} (relative $f$). The local shape of any plot of the pushforward diffeology is 
$f\circ p$ for some plot $p$ of $X$. 

All subsets, quotients, disjoint unions, and direct products of diffeological spaces, as well as spaces of maps between pairs of such, are canonically diffeological spaces. In particular, any subset $Y$ of a 
diffeological space $X$ carries the \textbf{subset diffeology} that is the set of all plots of $X$ with range wholly contained in $Y$. For any equivalence relation $\sim$ on $X$, the quotient $X/\sim$ has the 
\textbf{quotient diffeology} which is the pushforward of the diffeology of $X$ by the quotient projection. If $X$ and $Y$ are two diffeological spaces, and $C^{\infty}(X,Y)$ denotes the set of all smooth maps 
between them, then it is also a diffeological space, for the canonically chosen \textbf{functional diffeology}; this is the coarsest diffeology such that the evaluation map $(f,x)\mapsto f(x)$ is smooth (we refer 
the reader to \cite{iglesiasBook} for other definitions, and to \cite{wu} for definitions of constructions regarding diffeological vector spaces).

The operation of \textbf{diffeological gluing} is a simple way to obtain non-standard diffeological spaces, perhaps out of standard building blocks. It mimics in fact the usual operation of topological gluing. Its 
precise definition is as follows.

\begin{defn}\label{gluing:defn}
Let $X_1$ and $X_2$ be two diffeological spaces, and let $f:X_1\supseteq Y\to X_2$ be a map smooth for the subset diffeology on $Y$. The result of the \textbf{diffeological gluing of $X_1$ to $X_2$ 
along $f$} is the space
$$X_1\cup_f X_2:=(X_1\sqcup X_2)/\sim,\mbox{ where }X_1\supseteq Y\ni y\sim f(y)\in X_2,$$ endowed with the quotient diffeology of the disjoint union diffeology  on $X_1\sqcup X_2$. This diffeology 
is called the \textbf{gluing diffeology}.
\end{defn}

The gluing diffeology is a rather weak diffeology; for instance, if $X_1=X_2=\matR$ and $f:\{0\}\to\{0\}$, the resulting space $X_1\cup_f X_2$ can be smoothly identified with the union of the two 
coordinate axes in $\matR^2$. The gluing diffeology turns out to be strictly finer than the subset diffeology relative to the standard diffeology on $\matR^2$, see Example 2.67 in \cite{watts}.

For technical reasons, the following \textbf{standard inductions} are useful. Let 
$$i_1:(X_1\setminus Y)\hookrightarrow X_1\sqcup X_2\to X_1\cup_f X_2\,\,\,\mbox{ and }\,\,\,i_2:X_2\hookrightarrow X_1\sqcup X_2\to X_1\cup_f X_2$$ be given by the obvious inclusions into the 
disjoint union and by the quotient projection. These maps turn out to be inductions, with disjoint images that cover $X_1\cup_f X_2$. This is one reason why they are useful for defining maps from/into 
$X_1\cup_f X_2$.

The notion of a \textbf{diffeological vector pseudo-bundle} (that is a partial case of the \emph{diffeological fibre bundle} \cite{iglFibre}, where the concept originally appears, and which goes under the 
name of \emph{regular vector bundle} in \cite{vincent} and that of \emph{diffeological vector space over $X$} in \cite{CWtangent}) differs from that of the usual smooth vector bundle under two respects. 
First, we obviously consider the diffeological notion of smoothness, and second, it is not required to be locally trivial. The precise definition is as follows.

\begin{defn}
Let $V$ and $X$ be two diffeological spaces. A \textbf{diffeological vector pseudo-bundle} is a smooth surjective map $\pi:V\to X$ such that for every $x\in X$ the pre-image $\pi^{-1}(x)$ carries a 
vector space structure, and the following are satisfied:
\begin{enumerate}
\item The corresponding addition map is smooth as a map $V\times_X V\to V$ for the subset diffeology on $V\times_X V\subset V\times V$ and the product diffeology on $V\times V$;
\item The scalar multiplication is smooth as a map $\matR\times V\to V$, for the standard diffeology on $\matR$ and the product diffeology on $\matR\times V$;
\item The zero section is smooth as a map $X\to V$.
\end{enumerate}
\end{defn}

All the usual operations, such as taking sub-bundles, quotient bundles, direct sums, tensor products, and taking dual pseudo-bundles are defined for diffeological vector pseudo-bundles; see \cite{vincent} for 
the original source, \cite{wu} for another treatment of the tensor product (of vector spaces), and \cite{pseudobundles} for some details, as well as the behavior of the operations under diffeological gluing. Since 
diffeological vector spaces in general do not carry smooth scalar product, diffeological vextor pseudo-bundles might not have proper counterparts of Riemannian metrics. The closest possible substitute is the 
simple concept of a \textbf{pseudo-metric} (see \cite{pseudobundles} and \cite{pseudometric-pseudobundle} for details).

\begin{defn}
Let $\pi:V\to X$ be a finite-dimensional diffeological vector pseudo-bundle. A \textbf{pseudo-metric} on it is a smooth map $g:X\to V^*\otimes V^*$ such that for all $x\in X$ the bilinear form $g(x)$ 
is a symmetric semi-definite positive form of rank $\dim((\pi^{-1}(x))^*)$.
\end{defn}

If a given pseudo-bundle $\pi:V\to X$ can be endowed with a pseudo-metric $g$ then, as in the standard case, there is the standard \textbf{pairing map} that we denote $\Phi_g$. It is given as usual by 
$\Phi_g(v)(\cdot)=g(\pi(v))(v,\cdot)$ and is smooth, fibrewise linear, and surjective onto $V^*$. We can thus define a map $g^*:X\to V^{**}\otimes V^{**}$ whose value $g^*(x)$ for any $x\in X$ is a 
pseudo-metric on $\pi^{-1}(x)$. This map is determined by 
$$g^*(x)(\Phi_g(v),\Phi_g(w))=g(x)(v,w),$$ which is well-defined because $\Phi_g$ is surjective and its kernel obviously consists of the isotropic elements in the fibres. On the other hand, $\Phi_g$ is in general 
not invertible.

\subsection{The pseudo-bundle of diffeological $1$-forms $\Lambda^1(X)$}

There exists a rather well-developed theory of differential forms on diffeological spaces (a very recent exposition, although it is not the original source, can be found in \cite{iglesiasBook}); we recall it 
for $1$-forms.

\subsubsection{The construction}

A \textbf{differential $1$-form} on a diffeological space $X$ is any collection $\{\omega(p)\}$, where $p:U\to X$ runs over all plots of $X$ and $\omega(p)\in C^{\infty}(U,\Lambda^1(U))$ is a usual 
differential $1$-form on the domain $U$, that satisfies the following \textbf{smooth compatibility condition for differential forms}. For any two plots $p:U\to X$ and $q:U'\to X$ such that there exists 
a smooth in the usual sense map $F:U'\to U$ with the property that $q=p\circ F$, we have $\omega(q)=F^*(\omega(p))$, where $F^*$ is the usual pullback map. An instance of a diffeological differential 
form is the \textbf{differential} of a (diffeologically) smooth function $f:X\to\matR$ which is denoted by the usual symbol $df$ and is defined by $df(p)=d(f\circ p)$ for each and every plot $p$ of $X$.

The set of all differential $1$-forms on a given diffeological space $X$ is denoted by $\Omega^1(X)$. It has a natural vector space structure, where the addition and scalar multiplication are given 
pointwise, and a standard diffeology termed \textbf{functional diffeology} (see \cite{iglesiasBook}). The pseudo-bundle $\Lambda^1(X)$ is obtained as a quotient of the direct product $X\times\Omega^1(X)$, 
endowed with the product diffeology and viewed as the trivial bundle over $X$ (via the projection onto the first coordinate). The sub-bundle over which the quotient is taken is that given by the union 
$\bigcup_{x\in X}\{x\}\times\Omega_x^1(X)$ of all subspaces $\Omega_x^1(X)$ of so-called \textbf{vanishing forms}. A form $\omega\in\Omega^1(X)$ is said to be \textbf{vanishing at a given point $x\in X$} if 
for every plot $p:U\to X$ of $X$ such that $U\ni 0$ and $p(0)=x$ we have $\omega(p)(0)=0$. The set $\Omega_x^1(X)$ of all forms vanishing at $x$ is a vector subspace of $\Omega^1(X)$ and therefore the 
collection $\bigcup_{x\in X}\{x\}\times\Omega_x^1(X)$ of all of them is a sub-bundle of $X\times\Omega^1(X)\to X$ (for the subset diffeology). The pseudo-bundle $\Lambda^1(X)$ is the corresponding quotient 
pseudo-bundle,
$$\Lambda^1(X):=\left(X\times\Omega^1(X)\right)/\left(\bigcup_{x\in X}\{x\}\times\Omega_x^1(X)\right).$$ Its fibre at $x$ is $\Lambda_x^1(X)=\Omega^1(X)/\Omega_x^1(X)$.

\subsubsection{Diffeological $1$-forms and gluing}

The behavior of diffeological $1$-forms under gluing depends much on the properties of $f$. In particular, we will always assume that $f$ is a diffeomorphism satisfying some other technical conditions (see 
below), although some statements can be made without this assumption, such as the one describing the structure of $\Omega^1(X_1\cup_f X_2)$ (see \cite{forms-gluing}). The statement of this description, 
as well as that of the pseudo-bundle $\Lambda^1(X_1\cup_f X_2)$ relative to $\Lambda^1(X_1)$ and $\Lambda^1(X_2)$, requires some additional notions.

\begin{defn}\label{compatible:forms:defn}
Let $X_1$ and $X_2$ be two diffeological spaces, and let $f:X_1\supseteq Y\to X_2$ be a smooth map. Two forms $\omega_1\in\Omega^1(X_1)$ and $\omega_2\in\Omega^1(X_2)$ are called 
\textbf{compatible} if for every plot $p_1$ of the subset diffeology on $Y$ we have $\omega_1(p_1)=\omega_2(f\circ p_1)$.
\end{defn}

The notion of compatibility of forms in $\Omega^1(X_1)$ and $\Omega^1(X_2)$ can then be extended one of certain elements of $\Lambda^1(X_1)$ and $\Lambda^1(X_2)$.

\begin{defn}
Let $y\in Y$. Two cosets $\omega_1+\Omega_y^1(X_1)\in\Lambda_y^1(X_1)$ and $\omega_2+\Omega_{f(y)}^1(X_2)\in\Lambda_{f(y)}^1(X_2)$ are said to be \textbf{compatible} if for any 
$\omega_1'\in\Omega_y^1(X_1)$ and for any $\omega_2'\in\Omega_{f(y)}^1(X_2)$ the forms $\omega_1+\omega_1'$ and $\omega_2+\omega_2'$ are \textbf{compatible}.  
\end{defn}

We define next the conditions on the gluing map $f$ needed in most of our statements. Let $i:Y\hookrightarrow X_1$ and $j:f(Y)\hookrightarrow X_2$ be the natural inclusions; consider the 
corresponding pullback maps $i^*:\Omega^1(X_1)\to\Omega^1(Y)$ and $j^*:\Omega^1(X_2)\to\Omega^1(f(Y))$. The condition that we will most frequently use is that 
$$i^*(\Omega^1(X_1))=(f^*j^*)(\Omega^1(X_2)).$$ This condition ensures, in particular, that every fibre of $\Lambda^1(X_1\cup_f X_2)$ can be identified with either a fibre of one of $\Lambda^1(X_1)$, 
$\Lambda^1(X_2)$, or with a subspace of their direct sum, as Theorem \ref{individual:fibres:of:lambda:thm} below shows. 

\begin{thm}\label{individual:fibres:of:lambda:thm} \emph{(\cite{forms-gluing})}
Let $X_1$ and $X_2$ be two diffeological spaces, let $f:X_1\supseteq Y\to X_2$ be a diffeomorphism of its domain with its image such that there is the equality $i^*(\Omega^1(X_1))=(f^*j^*)(\Omega^1(X_2))$, 
and let $x\in X_1\cup_f X_2$. Then
$$\Lambda_x^1(X_1\cup_f X_2)\cong\left\{\begin{array}{cl} 
\Lambda_{i_1^{-1}(x)}(X_1) & \mbox{if }x\in i_1(X_1\setminus Y) \\
\Lambda_{f^{-1}(i_2^{-1}(x))}^1(X_1)\oplus_{comp}\Lambda_{i_2^{-1}(x)}^1(X_2)& \mbox{if }x\in i_2(f(Y)) \\
\Lambda_{i_2^{-1}(x)}(X_2) & \mbox{if }x\in i_2(X_2\setminus f(Y)) 
\end{array}\right.$$
\end{thm}

The diffeomorphisms of individual fibres mentioned in Theorem \ref{individual:fibres:of:lambda:thm} extend across certain collections of fibres of $\Lambda^1(X_1\cup_f X_2)$. This requires another, stronger 
condition, which is as follows. Let $\calD_1^{\Omega}$ be the diffeology on $\Omega^1(Y)$ that is the pushforward of the standard 
functional diffeology on $\Omega^1(X_1)$ by the map $i^*$, and let $\calD_2^{\Omega}$ be another diffeology on $\Omega^1(Y)$ and precisely the pushforward of the standard diffeology of $\Omega^1(X_2)$ 
by the map $f^*j^*$. Then both $\calD_1^{\Omega}$ and $\calD_2^{\Omega}$ are contained in the standard diffeology of $\Omega^1(Y)$, in general properly. The condition needed for extending the above 
diffeomorphisms of fibres is
$$\calD_1^{\Omega}=\calD_2^{\Omega};$$ we will use it in what follows, along with certain maps $\tilde{\rho}_1^{\Lambda}$ and $\tilde{\rho}_2^{\Lambda}$ that act
$$\tilde{\rho}_1^{\Lambda}:\Lambda^1(X_1\cup_f X_2)\supset(\pi^{\Lambda})^{-1}(i_1(X_1\setminus Y)\cup i_2(f(Y)))\to\Lambda^1(X_1)\,\,\,\mbox{ and}$$
$$\tilde{\rho}_2^{\Lambda}:\Lambda^1(X_1\cup_f X_2)\supset(\pi^{\Lambda})^{-1}(i_2(X_2))\to\Lambda^1(X_2).$$ The action of $\tilde{\rho}_1^{\Lambda}$ can be described by saying that on fibres over 
points in $i_1(X_1\setminus Y)$ it acts by identity, and on fibres over points in $i_2(f(Y))$ it acts as the projection onto the first factor; the map $\tilde{\rho}_2^{\Lambda}$ acts in a similar way. These maps 
allow, in particular, to characterize the diffeology of $\Lambda^1(X_1\cup_f X_2)$ as follows: 
\begin{itemize}
\item if $\calD_1^{\Omega}=\calD_2^{\Omega}$ and $U$ is a domain then $p:U\to\Lambda^1(X_1\cup_f X_2)$ is a plot of $\Lambda^1(X_1\cup_f X_2)$ if and only if $\tilde{\rho}_1^{\Lambda}\circ p$ and 
$\tilde{\rho}_2^{\Lambda}\circ p$ are both smooth wherever defined.
\end{itemize}
See \cite{forms-gluing} for details on these two maps.

Finally, let us consider pseudo-metrics on $\Lambda^1(X_1\cup_f X_2)$; more precisely, assuming that $\Lambda^1(X_1)$ and $\Lambda^1(X_2)$ admit pseudo-metrics, and that these pseudo-metrics 
are compatible in a certain natural sense (that we define immediately below), we describe a specific pseudo-metric on $\Lambda^1(X_1\cup_f X_2)$ induced by them.

\begin{defn}\label{compatible:pseudo-metrics:lambda:defn}
Let $g_1^{\Lambda}$ and $g_2^{\Lambda}$ be pseudo-metrics on $\Lambda^1(X_1)$ and $\Lambda^1(X_2)$ respectively. They are said to be \textbf{compatible} if for all $y\in Y$ and for any two 
compatible (in the sense of Definition \ref{compatible:forms:defn}) pairs $(\omega_1,\omega_2)$ and $(\mu_1,\mu_2)$, where $\omega_1,\mu_1\in(\pi_1^{\Lambda})^{-1}(y)$ and 
$\omega_2,\mu_2\in(\pi_2^{\Lambda})^{-1}(f(y))$ we have 
$$g_1^{\Lambda}(y)(\omega_1,\mu_1)=g_2^{\Lambda}(f(y))(\omega_2,\mu_2).$$
\end{defn} 

Then the following is true.

\begin{thm}\label{induced:pseudo-metric:on:glued:lambda:thm}
Let $X_1$ and $X_2$ be two diffeological spaces, and let $f:X_1\supseteq Y\to X_2$ be a diffeomorphism such that $\calD_1^{\Omega}=\calD_2^{\Omega}$. Suppose furthermore that $\Lambda^1(X_1)$ and 
$\Lambda^1(X_2)$ admit pseudo-metrics $g_1^{\Lambda}$ and $g_2^{\Lambda}$ compatible with $f$. Then $g^{\Lambda}$ defined by
$$g^{\Lambda}(x)(\cdot,\cdot)=\left\{\begin{array}{ll} 
g_1^{\Lambda}(i_1^{-1}(x))(\tilde{\rho}_1^{\Lambda}(\cdot),\tilde{\rho}_1^{\Lambda}(\cdot)) & \mbox{if }x\in i_1(X_1\setminus Y), \\
\frac12\left(g_1^{\Lambda}(f^{-1}(i_2^{-1}(x)))(\tilde{\rho}_1^{\Lambda}(\cdot),\tilde{\rho}_1^{\Lambda}(\cdot))+
g_2^{\Lambda}(i_2^{-1}(x))(\tilde{\rho}_2^{\Lambda}(\cdot),\tilde{\rho}_2^{\Lambda}(\cdot))\right) & \mbox{if }x\in i_2(f(Y)),\\
g_2^{\Lambda}(i_2^{-1}(x))(\tilde{\rho}_2^{\Lambda}(\cdot),\tilde{\rho}_2^{\Lambda}(\cdot)) & \mbox{if }x\in i_2(X_2\setminus f(Y)), 
\end{array}\right.$$
where each $(\cdot)$ stands for the value taken on an arbitrary element of the fibre at $x$ of $\Lambda^1(X_1\cup_f X_2)$, is a pseudo-metric on $\Lambda^1(X_1\cup_f X_2)$.
\end{thm}

\subsubsection{Compatibility of elements of $\Lambda^1(X_1)$ and $\Lambda^1(X_2)$ in terms of pullback maps $i_{\Lambda}^*$ and $j_{\Lambda^*}$}

We will need the following criterion of compatibility of elements of $\Lambda^1(X_1)$ and $\Lambda^1(X_2)$ with the gluing along a given diffeomorphism $f:X_1\supseteq Y\to X_2$. First, there is a 
well-defined pullback map $f_{\Lambda}^*:\Lambda^1(f(Y))\to\Lambda^1(Y)$ induced by the usual pullback map $f^*:\Omega^1(f(Y))\to\Omega^1(Y)$ and defined by setting, for any coset 
$\alpha_2=\omega_2+\Omega_{y'}^1(f(Y))\in\Lambda_{y'}^1(f(Y))$, that 
$$f_{\Lambda}^*(\alpha_2)=(f^{-1}(y'),f^*\omega_2).$$ Let now 
$$\pi_Y^{\Omega,\Lambda}:Y\times\Omega^1(Y)\to\Lambda^1(Y)\,\,\,\mbox{ and }\,\,\,\pi_{f(Y)}^{\Omega,\Lambda}:f(Y)\times\Omega^1(f(Y))\to\Lambda^1(f(Y))$$ be the defining projections of 
$\Lambda^1(Y)$ and $\Lambda^1(f(Y))$ respectively. Then (see \cite{connections}) the map $(i^{-1},i^*):i(Y)\times\Omega^1(X_1)\to Y\times\Omega^1(Y)$ descends to a map 
$i_{\Lambda}^*:\Lambda^1(X_1)\supset(\pi_1^{\Lambda})^{-1}(Y)\to\Lambda^1(Y)$ such that 
$\pi_Y^{\Omega,\Lambda}\circ(i^{-1},i^*)=i_{\Lambda}^*\circ\pi_1^{\Omega,\Lambda}|_{i(Y)\times\Omega^1(X_1)}$, and the map $(j^{-1},j^*):j(f(Y))\times\Omega^1(X_2)\to f(Y)\times\Omega^1(f(Y))$ descends 
to a map $j_{\Lambda}^*:\Lambda^1(X_2)\supset(\pi_2^{\Lambda})^{-1}(f(Y))\to\Lambda^1(f(Y))$ such that
$\pi_{f(Y)}^{\Omega,\Lambda}\circ(j^{-1},j^*)=j_{\Lambda}^*\circ\pi_2^{\Omega,\Lambda}|_{j(f(Y))\times\Omega^1(X_2)}$. Two elements $\alpha_1\in\Lambda_y^1(X_1)$ and 
$\alpha_2\in\Lambda_{f(y)}^1(X_2)$ are compatible if and only if 
$$i_{\Lambda}^*\alpha_1=f_{\Lambda}^*(j_{\Lambda}^*\alpha_2).$$

\subsection{Diffeological connections on diffeological pseudo-bundles}

A certain (preliminary, by the author's own admittance) notion of a connection appears in \cite{iglesiasBook}. Here we describe another approach that appears in \cite{connections} and is aimed to
resemble as much as possible to the standard notion.

\begin{defn}
Let $\pi:V\to X$ be a finite-dimensional diffeological vector pseudo-bundle, and let $C^{\infty}(X,V)$ be the space of its smooth sections. A \textbf{connection} on this pseudo-bundle is a smooth linear 
operator 
$$\nabla:C^{\infty}(X,V)\to C^{\infty}(X,\Lambda^1(X)\otimes V)$$ that satisfies the Leibnitz rule, \emph{i.e.}, for every function $f\in C^{\infty}(X,\matR)$ and for every section $s\in C^{\infty}(X,V)$ we have
$$\nabla(fs)=df\otimes s+f\nabla s.$$ 
\end{defn}

In this case, the differential $df$, which is usually  an element of $\Omega^1(X)$, is considered to be a section of $\Lambda^1(X)$ according to the following rule:
$$df:X\ni x\mapsto\pi^{\Omega,\Lambda}(x,df),$$ where the second symbol of $df$ stands for the diffeological $1$-form given by $df(p)=d(f\circ p)$.

\section{Diffeological Levi-Civita connections: definition and properties}

Let $X$ be a diffeological space such that $(\Lambda^1(X))^*$ admits a pseudo-metric, and let $g^{\Lambda^*}$ be a pseudo-metric on it. As in the standard case, a \textbf{diffeological Levi-Civita connection} 
on $X$ is any connection on $(\Lambda^1(X))^*$ that is compatible with $g^{\Lambda^*}$ and is symmetric, for the appropriately defined notions of symmetricity and of compatibility of a diffeological connection 
with a pseudo-metric.

\subsection{Levi-Civita connections on $(\Lambda^1(X))^*$}

This is the initial notion, although, as we will see shortly, the case of $(\Lambda^1(X))^*$ is equivalent to that of $\Lambda^1(X)$, as long as the latter has only finite-dimensional fibres. We first specify the 
action of a given section of $(\Lambda^1(X))^*$ on any given smooth function $X\to\matR$, with the result being again a smooth function. This allows to define, for any two sections of $(\Lambda^1(X))^*$, 
their Lie-bracket-like action on $C^{\infty}(X,\matR)$.

\paragraph{The action of $C^{\infty}(X,(\Lambda^1(X))^*)$ on $C^{\infty}(X,\matR)$} The action of $t\in C^{\infty}(X,(\Lambda^1(X))^*)$ on $C^{\infty}(X,\matR)$ is given by assigning to any
arbitrary smooth function $f:X\to\matR$ the function $t(f)\in C^{\infty}(X,\matR)$ given by
$$t(f):X\ni x\mapsto(t(x))(\pi^{\Omega,\Lambda}(x,df)),$$
where, recall, $df\in\Omega^1(X)$ is the differential of $f$ and $\pi^{\Omega,\Lambda}:X\times\Omega^1(X)\to\Lambda^1(X)$ is the defining quotient projection of $\Lambda^1(X)$. This is a smooth function $X\to\matR$, 
since it coincides with the pointwise evaluation of $t$ on $\pi^{\Omega,\Lambda}(x,df)$, which is a smooth section of $\Lambda^1(X)$. 

\paragraph{The Lie bracket} Given two sections $t_1,t_2\in C^{\infty}(X,(\Lambda^1(X))^*)$, it is now trivial to define their Lie bracket $[t_1,t_2]\in C^{\infty}(X,(\Lambda^1(X))^*)$; the defining relation fully 
mimics the standard case:
$$[t_1,t_2](s)=t_1(t_2(s))-t_2(t_1(s))\mbox{ for any }s\in C^{\infty}(X,\Lambda^1(X)).$$ The following statement is then obvious.

\begin{lemma}
For any $t_1,t_2\in C^{\infty}(X,(\Lambda^1(X))^*)$ the bracket $[t_1,t_2]\in C^{\infty}(X,(\Lambda^1(X))^*)$ is well-defined.
\end{lemma}

The operation $[,]$ has, of course, the same bilinearity and antisymmetricity properties, and satisfies the Jacobi identity.

\paragraph{The torsion tensor and symmetricity} Let us now consider the \textbf{torsion} of a diffeological connection $\nabla$ on $(\Lambda^1(X))^*$. We should first mention the notion of a \textbf{covariant 
derivative}, which is in fact a common one for diffeological connections, is always taken with respect to sections of $(\Lambda^1(X))^*$ and has the usual action (see \cite{connections}). The \textbf{torsion 
tensor} of a connection $\nabla^*$ on $(\Lambda^1(X))^*$ is then defined by the standard formula: for any two sections $t_1,t_2\in(\Lambda^1(X))^*$ we set
$$T(t_1,t_2):=\nabla^*_{t_1}t_2-\nabla^*_{t_2}t_1-[t_1,t_2].$$ 
Accordingly, a diffeological connection $\nabla^*$ on $(\Lambda^1(X))^*$ is said to be \textbf{symmetric} if its torsion is zero:
$$\nabla^*_{t_1}t_2-\nabla^*_{t_2}t_1=[t_1,t_2]\mbox{ for all }t_1,t_2\in C^{\infty}(X,(\Lambda^1(X))^*).$$ 

\paragraph{Compatibility with a pseudo-metric} This notion is defined analogously to that of a usual connection compatible with a given Riemannian metric and applies to any diffeological connection on a 
pseudo-bundle that admits a pseudo-metric (see \cite{connections}). We state this definition in the case that interests us here, that of connections and pseudo-metrics on $(\Lambda^1(X))^*$. 

\begin{defn}
Let $X$ be a diffeological space, let $\nabla^*$ be a diffeological connection on $(\Lambda^1(X))^*$, and let $g^{\Lambda^*}$ be a pseudo-metric on it. The connection $\nabla^*$ is said to be 
\textbf{compatible with $g^{\Lambda^*}$} if for any two sections $t_1,t_2\in C^{\infty}(X,(\Lambda^1(X))^*)$ we have 
$$d(g^{\Lambda^*}(t_1,t_2))=g^{\Lambda^*}(\nabla^*t_1,t_2)+g^{\Lambda^*}(t_1,\nabla^*t_2),$$ where the differential on the left is of course a section of $\Lambda^1(X)$ (rather than a differential form in 
$\Omega^1(X)$), and on the right the pseudo-metric $g^{\Lambda^*}$ is extended in the standard way to include the sections of $\Lambda^1(X)\otimes(\Lambda^1(X))^*$, that is, by setting 
$g^{\Lambda^*}(s\otimes t_1,t_2)=g^{\Lambda^*}(t_1,s\otimes t_2)=s\cdot g^{\Lambda^*}(t_1,t_2)$ for any section $s\in C^{\infty}(X,\Lambda^1(X))$.  
\end{defn}

Any connection on $((\Lambda^1(X))^*,g^{\Lambda^*})$, where $g^{\Lambda^*}$ is a pseudo-metric, that is symmetric and compatible with $g^{\Lambda^*}$ is called a \textbf{diffeological Levi-Civita 
connection on $X$}.

\subsection{Levi-Civita connections on $\Lambda^1(X)$}

We now consider the case of the pseudo-bundle $\Lambda^1(X)$. It turns out that, if $\Lambda^1(X)$ has only finite-dimensional fibres, then this case is fully equivalent to that of $(\Lambda^1(X))^*$.

\subsubsection{The pseudo-bundles $\Lambda^1(X)$ and $(\Lambda^1(X))^*$ relative to each other} 

Let $\Phi_{g^{\Lambda}}:\Lambda^1(X)\to(\Lambda^1(X))^*$ be the pairing map corresponding to the pseudo-metric $g^{\Lambda}$. Recall that it is defined by
$$\Phi_{g^{\Lambda}}(\alpha)(\cdot)=g^{\Lambda}(\pi^{\Lambda}(\alpha))(\alpha,\cdot)$$ for any $\alpha\in\Lambda^1(X)$. We make the following, more general, observation first.

\begin{lemma}
Let $\pi:V\to X$ be a finite-dimensional diffeological vector pseudo-bundle that admits a pseudo-metric $g$. Then the corresponding pairing map $\Phi_g$ is a subduction onto $V^*$.
\end{lemma}

\begin{proof}
The diffeology of any dual pseudo-bundle $V^*$ is defined as the finest diffeology such that all the evaluation functions $V^*\times V\to\matR$ are smooth. Let $\calD'$ be the pushforward of the diffeology 
of $V$ by the pairing map $\Phi_g$. Any plot of $\calD'$ locally has form $U\ni u\mapsto g(\pi(p(u)))(p(u),\cdot)\in V^*$ for some plot $p$ of $V$. Let $q:U'\to V$ be any other plot such that the set 
$\{(u,u')\,|\,\pi(p(u))=\pi(q(u'))\}$ is non-empty. The corresponding evaluation function, defined on this set, has form $(u,u')\mapsto g(p(u))(p(u),q(u'))$, and this is smooth by the definition of a pseudo-metric. 
This precisely implies that $\calD'$ coincides with the dual pseudo-bundle diffeology.
\end{proof}

We now turn to the implications of $\Lambda^1(X)$ having only finite-dimensional fibres.

\begin{lemma}
Let $X$ be such that all fibres of $\Lambda^1(X)$ have finite dimension. Then the subset diffeology on each fibre $\Lambda_x^1(X)$ is standard.
\end{lemma}

\begin{proof}
Let $x\in X$ be a point, and let $\omega_1+\Omega_x^1(X),\ldots,\omega_n+\Omega_x^1(X)$ be its basis. Let $q:U\to\Lambda_x^1(X)$ be a plot for the subset diffeology on the fibre of $\Lambda^1(X)$ at 
$x$. We can then write in the form $q(u)=q_1(u)(\omega_1+\Omega_x^1(X))+\ldots+q_n(u)(\omega_n+\Omega_x^1(X))$, where $q_i$ are some functions $U\to\matR$. The claim is then equivalent to each 
 $q_i$ being an ordinary smooth function.

Assume that $U$ is small so that $q=\pi^{\Omega,\Lambda}\circ p$ for some plot $p$ of $\Omega^1(X)$. We then have for all $u\in U$ that $p(u)=q_1(u)\omega_1+\ldots+q_n(u)\omega_n+\omega'(u)$, 
where $\omega'(u)\in\Omega_x^1(X)$. Let $p':U'\to X$ be a plot centered at $x$; we then have that 
$$(p(u))(p')(0)=q_1(u)\omega_1(p')(0)+\ldots+q_n(u)\omega_n(p')(0),$$ which is a restriction of the evaluation function for $p$ and $p'$ defined on $U\times U'$, to its smaller subset 
$U\times\{0\}\subset U\times U'$. Furthermore, by definition of the functional diffeology on $\Omega^1(X)$, this evaluation is an ordinary smooth section of the bundle of differential forms over $U\times U'$. 
It then follows that the coefficients of its restriction are ordinary smooth functions, as wanted.
\end{proof}

Recall that if all fibres of $\Lambda^1(X)$ are standard then any pseudo-metric on it is a fibrewise scalar product, and the corresponding pairing map is a diffeomorphism on each fibre. We have the following 
statement.

\begin{cor}\label{lambda:and:dual:lambda:are:diffeomorphic:cor}
Let $X$ be a diffeological space such that $\Lambda^1(X)$ has only finite-dimensional fibres and admits a pseudo-metric $g^{\Lambda}$. Then the pseudo-bundles $\Lambda^1(X)$ and $(\Lambda^1(X))^*$ 
are diffeomorphic.
\end{cor}

\begin{proof}
The desired diffeomorphism is given precisely by the map $\Phi_{g^{\Lambda}}$.
\end{proof}

\subsubsection{Defining Levi-Civita connections on $\Lambda^1(X)$} 

Let $X$ be such that $\Lambda^1(X)$ admits pseudo-metrics, and let $g^{\Lambda}$ be a pseudo-metric on it. As we have seen in the previous 
section, the pairing map $\Phi_{g^{\Lambda}}$ is a diffeomorphism $\Lambda^1(X)\to(\Lambda^1(X))^*$. This allows us to define first covariant derivatives along sections of $\Lambda^1(X)$ and then the 
symmetricity of a connection on $\Lambda^1(X)$. 

Let $\nabla$ be a connection on $\Lambda^1(X)$, and let $s\in C^{\infty}(X,\Lambda^1(X))$ be a section. There is a usual notion of a covariant derivative, of any diffeological connection, along a section of 
$(\Lambda^1(X))^*$. This makes the following definition obvious.

\begin{defn}
The \textbf{covariant derivative of $\nabla$ along $s$} is its covariant derivative along $\Phi_{g^{\Lambda}}\circ s$,
$$\nabla_s=\nabla_{(\Phi_{g^{\Lambda}})\circ s}.$$
\end{defn}

Let $s_1,s_2\in C^{\infty}(X,\Lambda^1(X))$. Their \textbf{Lie bracket} $[s_1,s_2]\in C^{\infty}(X,\Lambda^1(X))$ is defined by setting
$$[s_1,s_2]=\Phi_{g^{\Lambda}}^{-1}\circ[\Phi_{g^{\Lambda}}\circ s_1,\Phi_{g^{\Lambda}}\circ s_2],$$ where on the right we have the already-defined Lie bracket of the two sections of $(\Lambda^1(X))^*$.

\begin{defn}
A connection $\nabla$ on $\Lambda^1(X)$ is called \textbf{symmetric} if for any two sections $s_1,s_2\in C^{\infty}(X,\Lambda^1(X))$ we have
$$\nabla_{s_1}s_2-\nabla_{s_2}s_1=[s_1,s_2].$$ A symmetric connection $\nabla$ on $\Lambda^1(X)$ equipped with a pseudo-metric $g^{\Lambda}$ is called a \textbf{Levi-Civita connection} if it is 
compatible with $g^{\Lambda}$.
\end{defn}

\subsection{Uniqueness of Levi-Civita connections} 

The uniqueness property for Levi-Civita connections on diffeological spaces $X$ (with fibrewise finite-dimensional $\Lambda^1(X)$) holds just as it does for (finite-dimensional) smooth manifolds; in fact, it is 
established by exactly the same reasoning.

\begin{thm}\label{symm:comp:connection:is:unique:thm}
Let $X$ be a diffeological space such that $\Lambda^1(X)$ has finite-dimensional fibres only. Suppose furthermore that its dual pseudo-bundle $(\Lambda^1(X))^*$ is endowed with
a pseudo-metric $g^{\Lambda^*}$. Then there exists at most one connection $\nabla$ on $(\Lambda^1(X))^*$ that is symmetric and compatible with $g^{\Lambda^*}$.
\end{thm}

\begin{proof}
Suppose that there exists a connection $\nabla$ on $(\Lambda^1(X))^*$ that possesses the two properties indicated; let us show that it is unique. Let
$t_1,t_2,t_3\in C^{\infty}(X,(\Lambda^1(X))^*)$. By the assumption of compatibility of $\nabla$ with $g^{\Lambda^*}$ we have
$$d(g^{\Lambda^*}(t_2,t_3))=g^{\Lambda^*}(\nabla t_2,t_3)+g^{\Lambda^*}(t_2,\nabla t_3)\Rightarrow
t_1(g^{\Lambda^*}(t_2,t_3))=t_1(g^{\Lambda^*}(\nabla t_2,t_3))+t_1(g^{\Lambda^*}(t_2,\nabla t_3))\Rightarrow$$
$$\Rightarrow t_1(g^{\Lambda^*}(t_2,t_3))=g^{\Lambda^*}(\nabla_{t_1} t_2,t_3)+g^{\Lambda^*}(t_2,\nabla_{t_1} t_3),$$ where the left-hand term $t_1(g^{\Lambda^*}(t_2,t_3))$ stands for
the action of $t_1\in C^{\infty}(X,(\Lambda^1(X))^*)$. We then obtain the same equalities for $t_2,t_3,t_1$ and for $t_3,t_1,t_2$. Altogether, we have
$$\left\{\begin{array}{l}
t_1(g^{\Lambda^*}(t_2,t_3))=g^{\Lambda^*}(\nabla_{t_1} t_2,t_3)+g^{\Lambda^*}(t_2,\nabla_{t_1} t_3) \\
t_2(g^{\Lambda^*}(t_3,t_1))=g^{\Lambda^*}(\nabla_{t_2} t_3,t_1)+g^{\Lambda^*}(t_3,\nabla_{t_2} t_1) \\
t_3(g^{\Lambda^*}(t_1,t_2))=g^{\Lambda^*}(\nabla_{t_3} t_1,t_2)+g^{\Lambda^*}(t_1,\nabla_{t_3} t_2),
\end{array}\right.$$
from which by the standard calculation we obtain
\begin{flushleft}
$t_1(g^{\Lambda^*}(t_2,t_3))+t_2(g^{\Lambda^*}(t_3,t_1))-t_3(g^{\Lambda^*}(t_1,t_2))=$
\end{flushleft}
$$=g^{\Lambda^*}(\nabla_{t_1} t_2,t_3)+g^{\Lambda^*}(t_2,\nabla_{t_1} t_3)-g^{\Lambda^*}(\nabla_{t_3} t_1,t_2)-$$
$$-g^{\Lambda^*}(t_1,\nabla_{t_3} t_2)+g^{\Lambda^*}(\nabla_{t_2} t_3,t_1)+g^{\Lambda^*}(t_3,\nabla_{t_2} t_1)=$$
\begin{flushright}
$=g^{\Lambda^*}(\nabla_{t_1} t_2+\nabla_{t_2}t_1,t_3)+g^{\Lambda^*}(\nabla_{t_1}t_3-\nabla_{t_3}t_1,t_2)+g^{\Lambda^*}(\nabla_{t_2}t_3-\nabla_{t_3}t_2,t_1)$.
\end{flushright}
By the standard reasoning, the symmetricity assumption on $\nabla$ implies that
$$\nabla_{t_1}t_3-\nabla_{t_3}t_1=[t_1,t_3],\,\,\,\nabla_{t_2}t_3-\nabla_{t_3}t_2=[t_2,t_3],\,\,\,\mbox{ and }\,\,\,\nabla_{t_1} t_2+\nabla_{t_2}t_1=[t_1,t_2]+2\nabla_{t_2}t_1=[t_2,t_1]+2\nabla_{t_1}t_2.$$
Therefore we obtain
\begin{flushleft}
$t_1(g^{\Lambda^*}(t_2,t_3))+t_2(g^{\Lambda^*}(t_3,t_1))-t_3(g^{\Lambda^*}(t_1,t_2))+g^{\Lambda^*}([t_1,t_2],t_3)=$
\end{flushleft}
\begin{flushright}
$=g^{\Lambda^*}([t_1,t_3],t_2)+g^{\Lambda^*}([t_2,t_3],t_1)-g^{\Lambda^*}([t_1,t_2],t_3)+2g^{\Lambda^*}(\nabla_{t_1}t_2,t_3)$.
\end{flushright}
Using the antisymmetricity of the Lie bracket, we also obtain its equivalent form, which is
\begin{flushleft}
$t_1(g^{\Lambda^*}(t_2,t_3))+t_2(g^{\Lambda^*}(t_3,t_1))-t_3(g^{\Lambda^*}(t_1,t_2))+$
\end{flushleft}
\begin{flushright}
$+g^{\Lambda^*}([t_1,t_2],t_3)+g^{\Lambda^*}([t_3,t_1],t_2)-g^{\Lambda^*}([t_2,t_3],t_1)=2g^{\Lambda^*}(\nabla_{t_1}t_2,t_3)$.
\end{flushright}
Since any pseudo-metric on $(\Lambda^1(X))^*$ is non-degenerate, we obtain the desired conclusion.
\end{proof}

\begin{rem}
The same uniqueness result obviously holds for Levi-Civita connections on $\Lambda^1(X)$, due to the existence of the pairing map diffeomorphism.
\end{rem}

\section{Levi-Civita connections on $\Lambda^1(X_1)$ and $\Lambda^1(X_2)$, and diffeological gluing}

In this section we consider the behavior of connections on $\Lambda^1(X_1)$ and $\Lambda^1(X_2)$, in particular, of the Levi-Civita connections with respect to the diffeological gluing. We show that certain 
pairs of connections on $\Lambda^1(X_1)$ and $\Lambda^1(X_2)$, determined by a suitable compatibility notion, give rise to a well-defined connection on $\Lambda^1(X_1\cup_f X_2)$. In particular, the 
Levi-Civita connections on $\Lambda^1(X_1)$ and $\Lambda^1(X_2)$ determined with respect to compatible pseudo-metrics determine the Levi-Civita connection on $\Lambda^1(X_1\cup_f X_2)$.

\subsection{Compatibility of connections on $\Lambda^1(X_1)$ and $\Lambda^1(X_2)$}

We first treat the more general case of two given connections on $\Lambda^1(X_1)$ and $\Lambda^1(X_2)$, which do not have to be the Levi-Civita connections on them. In particular, the following notion 
of compatibility is generally applicable.

\begin{defn}\label{compatible:connections:lambda:defn}
Let $X_1$ and $X_2$ be two diffeological spaces, let $f:X_1\supseteq Y\to X_2$ be a diffeomorphism, and let $\nabla^1$ and $\nabla^2$ be connections on $\Lambda^1(X_1)$ and $\Lambda^1(X_2)$ 
respectively. We say that $\nabla^1$ and $\nabla^2$ are \textbf{compatible} if for all pairs of compatible sections $s_i\in C^{\infty}(X_i,\Lambda^1(X_i))$ and for all $y\in Y$ they satisfy
$$\left((i_{\Lambda}^*\otimes i_{\Lambda}^*)\circ(\nabla^1s_1)\right)(y)=\left(((f_{\Lambda}^*j_{\Lambda}^*)\otimes(f_{\Lambda}^*j_{\Lambda}^*))\circ(\nabla^2s_2)\right)(f(y)).$$ 
\end{defn}

The definition is designed so that over each point in the domain of gluing $i_2(f(Y))\subset X_1\cup_f X_2$ the sum 
$$(\nabla^1s_1)(y)\oplus(\nabla^2s_2)(f(y))\in\Lambda_y^1(X_1)\otimes\Lambda_y^1(X_1)\,\oplus\,\Lambda_{f(y)}^1(X_2)\otimes\Lambda_{f(y)}^1(X_2)$$ be an element of 
$$\Lambda_{i_2(f(y))}^1(X_1\cup_f X_2)\otimes\Lambda_{i_2(f(y))}^1(X_1\cup_f X_2)\cong
(\Lambda_y^1(X_1)\oplus_{comp}\Lambda_{f(y)}^1(X_2))\otimes(\Lambda_y^1(X_1)\oplus_{comp}\Lambda_{f(y)}^1(X_2)),$$ where the diffeomorphism is via the map
$(\tilde{\rho}_1^{\Lambda}\oplus\tilde{\rho}_2^{\Lambda})\otimes(\tilde{\rho}_1^{\Lambda}\oplus\tilde{\rho}_2^{\Lambda})$. The following statement is then immediate.

\begin{lemma}\label{compatible:connections:over:domain:of:gluing:lem}
Let $\nabla^1$ and $\nabla^2$ be compatible connections, and let $y\in Y$ be a point. Then $(\nabla^1s_1)(y)\oplus(\nabla^2s_2)(f(y))$ belongs to the range of the map 
$(\tilde{\rho}_1^{\Lambda}\oplus\tilde{\rho}_2^{\Lambda})\otimes(\tilde{\rho}_1^{\Lambda}\oplus\tilde{\rho}_2^{\Lambda})$. 
\end{lemma}

\subsection{The induced connection on $\Lambda^1(X_1\cup_f X_2)$}

We now show that any pair $\nabla^1,\nabla^2$ of compatible connections on $\Lambda^1(X_1)$ and $\Lambda^1(X_2)$ induces a well-defined connection $\nabla^{\cup}$ on 
$\Lambda^1(X_1\cup_f X_2)$.

\subsubsection{Splitting a section $s\in C^{\infty}(X_1\cup_f X_2,\Lambda^1(X_1\cup_f X_2))$}\label{splitting:sections:sect}

Let $s$ be a smooth section $X_1\cup_f X_2\to\Lambda^1(X_1\cup_f X_2)$. Recall that, since $f$ is assumed to be a diffeomorphism, the map $\tilde{i}_1:X_1\to X_1\cup_f X_2$ given by the composition 
of the inclusion $X_1\hookrightarrow X_1\sqcup X_2$ and the defining projection $X_1\sqcup X_2\to X_1\cup_f X_2$, is an induction. Denote
$$s_1:=\tilde{\rho}_1^{\Lambda}\circ s\circ\tilde{i}_1\,\,\mbox{ and }\,\,s_2:=\tilde{\rho}_2^{\Lambda}\circ s\circ i_2.$$ It is then clear that $s_1\in C^{\infty}(X_1,\Lambda^1(X_1))$ and 
$s_2\in C^{\infty}(X_2,\Lambda^1(X_2))$.

\subsubsection{The differential and gluing}

The behavior of the differential with respect to gluing is described in \cite{connections} (Corollary 5.11). The statement is as follows. 

\begin{prop}\label{differential:under:gluing:prop}
Let $X_1$ and $X_2$ be two diffeological spaces, and let $f:X_1\supseteq Y\to X_2$ be a diffeormophism such that $\calD_1^{\Omega}=\calD_2^{\Omega}$. Let $h:X_1\cup_f X_2\to\matR$ be a smooth 
function, and denote $h_1:=h\circ\tilde{i}_1:X_1\to\matR$ and $h_2:=h\circ i_2:X_2\to\matR$. Then its differential $dh\in\Lambda^1(X_1\cup_f X_2)$ is given by
$$dh(x)=\left\{\begin{array}{cl} 
(\tilde{\rho}_1^{\Lambda})^{-1}(dh_1(i_1^{-1}(x))) & \mbox{if }x\in i_1(X_1\setminus Y), \\
(\tilde{\rho}_1^{\Lambda}\oplus\tilde{\rho}_2^{\Lambda})^{-1}(dh_1(\tilde{i}_1^{-1}(x))+dh_2(i_2^{-1}(x))) & \mbox{if }x\in i_2(f(Y)), \\
(\tilde{\rho}_2^{\Lambda})^{-1}(dh_2(i_2^{-1}(x))) & \mbox{if }x\in i_2(X_2\setminus f(Y)).
\end{array}\right.$$
\end{prop}

\subsubsection{The connection $\nabla^{\cup}$}

We now define the induced operator
$$\nabla^{\cup}:C^{\infty}(X_1\cup_f X_2,\Lambda^1(X_1\cup_f X_2))\to C^{\infty}(X_1\cup_f X_2,\Lambda^1(X_1\cup_f X_2)\otimes\Lambda^1(X_1\cup_f X_2)).$$ Let 
$s\in C^{\infty}(X_1\cup_f X_2,\Lambda^1(X_1\cup_f X_2))$ be a section, and let $x\in X_1\cup_f X_2$ be a point. 

\begin{defn}
Set $\nabla^{\cup}$ to be the operator such that
\begin{flushleft}
$(\nabla^{\cup}s)(x)=$
\end{flushleft}
$$\left\{\begin{array}{cl}
(((\tilde{\rho}_1^{\Lambda})^{-1}\otimes(\tilde{\rho}_1^{\Lambda})^{-1})\circ(\nabla^1s_1))(i_1^{-1}(x)) & \mbox{if }x\in i_1(X_1\setminus Y), \\
((\tilde{\rho}_1^{\Lambda}\oplus\tilde{\rho}_2^{\Lambda})^{-1}\otimes(\tilde{\rho}_1^{\Lambda}\oplus\tilde{\rho}_2^{\Lambda})^{-1})\left((\nabla^1s_1)(\tilde{i}_1^{-1}(x))\oplus(\nabla^2s_2)(i_2^{-1}(x))\right) & 
\mbox{if }x\in i_2(f(Y)), \\
(((\tilde{\rho}_2^{\Lambda})^{-1}\otimes(\tilde{\rho}_2^{\Lambda})^{-1})\circ(\nabla^2s_2))(i_2^{-1}(x)) & \mbox{if }x\in i_2(X_2\setminus f(Y)).
\end{array}\right.$$
\end{defn}

We claim that the operator $\nabla^{\cup}$ thus defined is a diffeological connection on $\Lambda^1(X_1\cup_f X_2)$.

\begin{thm}
Let $X_1$ and $X_2$ be diffeological spaces, and let $f:X_1\supseteq Y\to X_2$ be a diffeomorphism such that $\calD_1^{\Omega}=\calD_2^{\Omega}$. Then $\nabla^{\cup}$ is a diffeological connection 
on $\Lambda^1(X_1\cup_f X_2)$.
\end{thm}

\begin{proof}
We need to check that $\nabla^{\cup}$ is well-defined as a map 
$C^{\infty}(X_1\cup_f X_2,\Lambda^1(X_1\cup_f X_2))\to C^{\infty}(X_1\cup_f X_2,\Lambda^1(X_1\cup_f X_2)\otimes\Lambda^1(X_1\cup_f X_2))$, that is, that $(\nabla^{\cup}s)(x)$ is always an element 
of $\Lambda^1(X_1\cup_f X_2)\otimes\Lambda^1(X_1\cup_f X_2)$, and that $\nabla^{\cup}s$ is smooth for any $s\in C^{\infty}(X_1\cup_f X_2,\Lambda^1(X_1\cup_f X_2))$, that it is additive and satisfies 
the Leibnitz rule, and finally, that it is smooth for the functional diffeologies on $C^{\infty}(X_1\cup_f X_2,\Lambda^1(X_1\cup_f X_2))$ and 
$C^{\infty}(X_1\cup_f X_2,\Lambda^1(X_1\cup_f X_2)\otimes\Lambda^1(X_1\cup_f X_2))$.

For $x\in i_1(X_1\setminus Y)$ the fibre $\left(\Lambda^1(X_1\cup_f X_2)\otimes\Lambda^1(X_1\cup_f X_2)\right)_x$ is identified to 
$\Lambda_{i_1^{-1}(x)}^1(X_1)\otimes\Lambda_{i_1^{-1}(x)}^1(X_1)$ via the map $\tilde{\rho}_1^{\Lambda}\otimes\tilde{\rho}_1^{\Lambda}$. Likewise, for $x\in i_2(X_2\setminus f(Y))$ such fibre is identified 
to $\Lambda_{i_2^{-1}(x)}^1(X_2)\otimes\Lambda_{i_2^{-1}(x)}^1(X_2)$ via the map $\tilde{\rho}_2^{\Lambda}\otimes\tilde{\rho}_2^{\Lambda}$. For such points the first statement is therefore obvious, while 
for $x\in i_2(f(Y))$ it follows from the compatibility notion for connections, see Lemma \ref{compatible:connections:over:domain:of:gluing:lem}.

Let us now check the smoothness of $\nabla^{\cup}s:X_1\cup_f X_2\to\Lambda^1(X_1\cup_f X_2)\otimes\Lambda^1(X_1\cup_f X_2)$ for an arbitrary section 
$s\in C^{\infty}(X_1\cup_f X_2,\Lambda^1(X_1\cup_f X_2))$. Let $p:U\to X_1\cup_f X_2$ be a plot. We can assume that $U$ is connected so that either there is a plot $p_1$ of $X_1$ such that 
$p=\tilde{i}_1\circ p_1$, or there is a plot $p_2$ of $X_2$ such that $p=i_2\circ p_2$. Since $f$ is a diffeomorphism, the two cases are symmetric, so it suffices to consider one of them, say, the former one. 

If $p=\tilde{i}_1\circ p_1$ then 
\begin{flushleft}
$(\nabla^{\cup}s)(p(u))=$
\end{flushleft}
$$\left\{\begin{array}{cl}
(((\tilde{\rho}_1^{\Lambda})^{-1}\otimes(\tilde{\rho}_1^{\Lambda})^{-1})\circ(\nabla^1s_1))(p_1(u)) & \mbox{for }u\mbox{ such that }p_1(u)\in X_1\setminus Y, \\
((\tilde{\rho}_1^{\Lambda}\oplus\tilde{\rho}_2^{\Lambda})^{-1}\otimes(\tilde{\rho}_1^{\Lambda}\oplus\tilde{\rho}_2^{\Lambda})^{-1})\left(((\nabla^1s_1)(p_1(u))\oplus(\nabla^2s_2)(f(p_1(u)))\right) & 
\mbox{for }u\mbox{ such that }p_1(u)\in Y. \\
\end{array}\right.$$ We need to show that this is a plot of $\Lambda^1(X_1\cup_f X_2)\otimes\Lambda^1(X_1\cup_f X_2)$. It is thus sufficient to show that its (partial) compositions with 
$\tilde{\rho}_1^{\Lambda}\otimes\tilde{\rho}_1^{\Lambda}$ and $\tilde{\rho}_2^{\Lambda}\otimes\tilde{\rho}_2^{\Lambda}$ are smooth maps. 

We have that 
$$(\tilde{\rho}_1^{\Lambda}\otimes\tilde{\rho}_1^{\Lambda})((\nabla^{\cup}s)(p(u)))=(\nabla^1s_1))(p_1(u))\,\,\mbox{ for all }u\in U.$$ Since $\nabla^1$ is a connection on $\Lambda^1(X_1)$ and $p_1$ is a 
plot of $X_1$, this is a plot of $\Lambda^1(X_1)\otimes\Lambda^1(X_1)$. Next, we need to show that $(\tilde{\rho}_2^{\Lambda}\otimes\tilde{\rho}_2^{\Lambda})\circ(\nabla^{\cup}s)\circ p$ is a smooth map 
on its domain of definition, that is, on $p_1^{-1}(Y)$ (the latter may not be a domain, so it is not in general a plot). This amounts to showing that for an ordinary smooth function $h:U'\to U$ taking values in 
$p_1^{-1}(Y)$ the following composition is a plot of $\Lambda^1(X_2)\otimes\Lambda^1(X_2)$:
$$(\tilde{\rho}_2^{\Lambda}\otimes\tilde{\rho}_2^{\Lambda})\circ(\nabla^{\cup}s)\circ p\circ h=(\nabla^2s_2)\circ(f\circ p_1\circ h).$$ It suffices to observe that by the properties of diffeologies $p_1\circ h$ is a 
plot for the subset diffeology on $Y$, hence $f\circ p_1\circ h$ is a plot of $f(Y)$, and $(\nabla^2s_2)\circ(f\circ p_1\circ h)$ is a plot of $\Lambda^1(X_2)\otimes\Lambda^1(X_2)$, since $\nabla^2$ is assumed 
to be a connection.

Since for any two sections $s,s'\in C^{\infty}(X_1\cup_f X_2,\Lambda^1(X_1\cup_f X_2))$ we have that $(s+s')_1=s_1+s_1' $ and $(s+s')_2=s_2+s_2'$ (see Lemma 2.6 in \cite{connections}), the additivity 
is obvious. Let us check the Leibnitz rule. Let $s\in C^{\infty}(X_1\cup_f X_2,\Lambda^1(X_1\cup_f X_2))$, and let $h:X_1\cup_f X_2\to\matR$ be smooth. Let $h_1$ and $h_2$ be as in Proposition 
\ref{differential:under:gluing:prop}. By Proposition 4.7 in \cite{pseudometric-pseudobundle}, we have that $(hs)_1=h_1s_1$ and $(hs)_2=h_2s_2$. Consider $(\nabla^{\cup}(hs))(x)$ for an arbitrary 
$x\in X_1\cup_f X_2$. 

If $x\in i_1(X_1\setminus Y)$ or $x\in i_2(X_2\setminus f(Y))$, the claim is obvious from Proposition \ref{differential:under:gluing:prop}. Suppose that $x\in i_2(f(Y))$. Then 
\begin{flushleft}
$(\nabla^{\cup}(hs))(x)=((\tilde{\rho}_1^{\Lambda}\oplus\tilde{\rho}_2^{\Lambda})^{-1}\otimes(\tilde{\rho}_1^{\Lambda}\oplus\tilde{\rho}_2^{\Lambda})^{-1})\left((\nabla^1(h_1s_1))(\tilde{i}_1^{-1}(x))
+(\nabla^2(h_2s_2))(i_2^{-1}(x))\right)=$
\end{flushleft}
$$=((\tilde{\rho}_1^{\Lambda}\oplus\tilde{\rho}_2^{\Lambda})^{-1}\otimes(\tilde{\rho}_1^{\Lambda}\oplus\tilde{\rho}_2^{\Lambda})^{-1})((dh_1\otimes s_1)(\tilde{i}_1^{-1}(x))+
(dh_2\otimes s_2)(i_2^{-1}(x))+$$
$$+(\nabla^1s_1)(\tilde{i}_1^{-1}(x))+(\nabla^2s_2)(i_2^{-1}(x)))=$$
\begin{flushleft}
$=((\tilde{\rho}_1^{\Lambda}\oplus\tilde{\rho}_2^{\Lambda})^{-1}\otimes(\tilde{\rho}_1^{\Lambda}\oplus\tilde{\rho}_2^{\Lambda})^{-1})\left((dh_1\otimes s_1)(f^{-1}(i_2^{-1}(x)))+
(dh_2\otimes s_2)(i_2^{-1}(x))\right)+(\nabla^{\cup}s)(x)=$
\end{flushleft}
\begin{flushright}
$=(\tilde{\rho}_1^{\Lambda}\oplus\tilde{\rho}_2^{\Lambda})^{-1}(dh_1(\tilde{i}_1^{-1}(x))+dh_2(i_2^{-1}(x)))\otimes s(x)+(\nabla^{\cup}s)(x),$
\end{flushright} as wanted.

It remains to show that $\nabla^{\cup}$ is smooth with respect to the functional diffeologies involved, that is, that for any plot $q:U'\to C^{\infty}(X_1\cup_f X_2,\Lambda^1(X_1\cup_f X_2))$ the assignment 
$u'\mapsto\nabla^{\cup}q(u')$ is a plot of $C^{\infty}(X_1\cup_f X_2,\Lambda^1(X_1\cup_f X_2)\otimes\Lambda^1(X_1\cup_f X_2))$. This, in turn, amounts to showing that for any plot $p:U\to X_1\cup_f X_2$ 
the map $(u,u')\to(\nabla^{\cup}q(u))(p(u'))$ is a plot of $\Lambda^1(X_1\cup_f X_2)\otimes\Lambda^1(X_1\cup_f X_2)$. Observing that $q(u)_1$ and $q(u)_2$ corresponding to each $q(u)$ are by 
construction plots $q_1$ and $q_2$ of $C^{\infty}(X_1,\Lambda^1(X_1))$ and $C^{\infty}(X_2,\Lambda^1(X_2))$, we proceed exactly as in the case of proving the smoothness of $\nabla^{\cup}s$. 

Say, for instance, that $p$ lifts to a plot $p_1$ of $X_1$; then it suffices to verify that, for any smooth function $h:U''\to p_1^{-1}(Y)\subseteq U'$, the maps 
$$(u,u')\mapsto(\nabla^1(q_1(u)))(p_1(u'))\,\,\mbox{ and }\,\,(u,u'')\mapsto(\nabla^2(q_2(u)))((f\circ p_1\circ h)(u''))$$ are plots of $\Lambda^1(X_1)\otimes\Lambda^1(X_1)$ and 
$\Lambda^1(X_2)\otimes\Lambda^1(X_2)$ respectively. This follows by all the same reasoning: $\nabla^1$ and $\nabla^2$ are connections by assumption, $p_1$ and $f\circ p_1\circ h$ are plots (of $X_1$ 
and $X_2$) by their choice, and $q_1$ and $q_2$ are plots of $C^{\infty}(X_1,\Lambda^1(X_1))$ and $C^{\infty}(X_2,\Lambda^1(X_2))$ by the remark already made. We thus conclude that $\nabla^{\cup}$ 
is indeed a diffeological connection on $\Lambda^1(X_1\cup_f X_2)$.
\end{proof}

\subsection{The connection induced by the Levi-Civita connections}

We prove here if the compatible connections $\nabla^1$ and $\nabla^2$ are Levi-Civita connections relative to compatible pseudo-metrics $g_1^{\Lambda}$ and $g_2^{\Lambda}$ then $\nabla^{\cup}$ 
induced by them is also the Levi-Civita connection with respect to the induced pseudo-metric $g^{\Lambda}$.

\subsubsection{If $\nabla^1$ and $\nabla^2$ are compatible with $g_1^{\Lambda}$ and $g_2^{\Lambda}$ then $\nabla^{\cup}$ is compatible with $g^{\Lambda}$}

Let $g_1^{\Lambda}$ and $g_2^{\Lambda}$ be compatible pseudo-metrics on $\Lambda^1(X_1)$ and $\Lambda^1(X_2)$. Suppose that these pseudo-bundles admit connections $\nabla^1$ and $\nabla^2$
that are, on one hand, compatible with each other, and, on the other hand, each is compatible with the appropriate pseudo-metric. Then $\Lambda^1(X_1\cup_f X_2)$ comes endowed with both the induced 
pseudo-metric $g^{\Lambda}$ and the induced connection $\nabla^{\cup}$. It would be natural to expect that also $\nabla^{\cup}$ be compatible with $g^{\Lambda}$.

\begin{thm}
Let $X_1$ and $X_2$ be diffeological spaces, and let $f:X_1\supseteq Y\to X_2$ be a gluing diffeomorphism such that $\calD_1^{\Omega}=\calD_2^{\Omega}$. Let $g_1^{\Lambda}$ be a pseudo-metric on 
$\Lambda^1(X_1)$, and let $\nabla^1$ be a connection on $\Lambda^1(X_1)$ compatible with $g_1^{\Lambda}$. Let $g_2^{\Lambda}$ be a pseudo-metric on $\Lambda^1(X_2)$, and let $\nabla^2$ be a 
compatible connection. Finally, assume that $g_1^{\Lambda}$ and $g_2^{\Lambda}$ are compatible with the gluing along $f$, and that also $\nabla^1$ and $\nabla^2$ are compatible. Then $\nabla^{\cup}$ 
is compatible with $g^{\Lambda}$.
\end{thm}

\begin{proof}
Let $s,r\in C^{\infty}(X_1\cup_f X_2,\Lambda^1(X_1\cup_f X_2))$. We need to check that 
$$d(g^{\Lambda}(s,t))=g^{\Lambda}(\nabla^{\cup}s,t)+g^{\Lambda}(s,\nabla^{\cup}t).$$ Let us consider this equality pointwise. There are three cases, that of a point in $i_1(X_1\setminus Y)$, one in 
$i_2(X_2\setminus f(Y))$, and finally, that of $x\in i_2(f(Y))$. Notice that the function $x\mapsto g^{\Lambda}(x)(s(x),t(x))$ splits (in the sense explained in Section \ref{splitting:sections:sect}) precisely into 
the pairs of functions $\tilde{i}_1^{-1}(x)\mapsto g_1^{\Lambda}(\tilde{i}_1^{-1}(x))(s_1(\tilde{i}_1^{-1}(x)),t_1(\tilde{i}_1^{-1}(x)))$ and 
$i_2^{-1}(x)\mapsto g_2^{\Lambda}(i_2^{-1}(x))(s_2(i_2^{-1}(x)),t_2(i_2^{-1}(x)))$.

Let $x\in i_1(X_1\setminus Y)$. In this case $g^{\Lambda}(x)(s(x),t(x))=g_1^{\Lambda}(i_1^{-1}(x))(s_1(i_1^{-1}(x)),t_1(i_1^{-1}(x)))$, and therefore
$$d(g^{\Lambda}(s,t))(x)=(\tilde{\rho}_1^{\Lambda})^{-1}(d(g_1^{\Lambda}(s_1,t_1))(i_1^{-1}(x))).$$ Furthermore,
$$\begin{array}{rcl} 
(\nabla^{\cup}s)(x) & = & (((\tilde{\rho}_1^{\Lambda})^{-1}\otimes(\tilde{\rho}_1^{\Lambda})^{-1})\circ(\nabla^1s_1))(i_1^{-1}(x)), \\
g^{\Lambda}(x)((\nabla^{\cup}s)(x),t(x)) & = & g_1^{\Lambda}(i_1^{-1}(x))((\nabla^1s_1)(i_1^{-1}(x)),t_1(i_1^{-1}(x))), \\
g^{\Lambda}(x)(s(x),(\nabla^{\cup}t)(x)) & = & g_1^{\Lambda}(i_1^{-1}(x))(s_1(i_1^{-1}(x)),(\nabla^1t_1)(i_1^{-1}(x))).
\end{array}$$ Since the equality $d(g_1^{\Lambda}(s_1,t_1))=g_1^{\Lambda}(\nabla^1s_1,t_1)+g_1^{\Lambda}(s_1,\nabla^1t_1)$ holds by assumption, we immediately obtain the desired conclusion.

The case of $x\in i_2(X_2\setminus f(Y))$ is fully analogous to that of $x\in i_1(X_1\setminus Y)$, so let now $x\in i_2(f(Y))$. In this case
$$g^{\Lambda}(x)(s(x),t(x))=\frac12g_1^{\Lambda}(\tilde{i}_1^{-1}(x))(s_1(\tilde{i}_1^{-1}(x)),t_1(\tilde{i}_1^{-1}(x)))+\frac12g_2^{\Lambda}(i_2^{-1}(x))(s_2(i_2^{-1}(x)),t_2(i_2^{-1}(x))),$$ hence 
$$d(g^{\Lambda}(s,t))(x)=\frac12(\tilde{\rho}_1^{\Lambda}+\tilde{\rho}_2^{\Lambda})^{-1}(dg_1^{\Lambda}(s_1,t_1)(\tilde{i}_1^{-1}(x))\oplus dg_2^{\Lambda}(s_2,t_2)(i_2^{-1}(x))).$$
For the right-hand side terms we have
$$(\nabla^{\cup}s)(x)=
((\tilde{\rho}_1^{\Lambda}+\tilde{\rho}_2^{\Lambda})^{-1}\otimes(\tilde{\rho}_1^{\Lambda}+\tilde{\rho}_2^{\Lambda})^{-1})((\nabla^1s_1)(\tilde{i}_1^{-1}(x))\oplus(\nabla^2s_2)(i_2^{-1}(x))),$$
\begin{flushleft}
$g^{\Lambda}(x)((\nabla^{\cup}s)(x),t(x))=$ 
\end{flushleft}
\begin{flushright}
$\frac12(g_1^{\Lambda}(\tilde{i}_1^{-1}(x))((\nabla^1s_1)(\tilde{i}_1^{-1}(x)),t_1(\tilde{i}_1^{-1}(x)))+g_2^{\Lambda}(i_2^{-1}(x))((\nabla^2s_2)(i_2^{-1}(x)),t_2(i_2^{-1}(x)))),$
\end{flushright}
\begin{flushleft}
$g^{\Lambda}(x)(s(x),(\nabla^{\cup}t)(x))=$
\end{flushleft}
\begin{flushright}
$\frac12(g_1^{\Lambda}(\tilde{i}_1^{-1}(x))(s_1(\tilde{i}_1^{-1}(x)),(\nabla^1t_1)(\tilde{i}_1^{-1}(x)))+g_2^{\Lambda}(i_2^{-1}(x))(s_2(i_2^{-1}(x)),(\nabla^2t_2)(i_2^{-1}(x)))).$
\end{flushright}
The numerical coefficient $\frac12$ cancels out on the two sides of the desired equality, and we deduce it from the assumptions of compatibility of $\nabla^1$ with $g_1^{\Lambda}$, and of $\nabla^2$ with 
$g_2^{\Lambda}$. All cases having been considered, we obtain the final claim.
\end{proof}

\subsubsection{Covariant derivatives and gluing}

Let $s$ be an arbitrary section of $\Lambda^1(X_1\cup_f X_2)$. Recall the sections $s_1\in C^{\infty}(X_1,\Lambda^1(X_1))$ and $s_2\in C^{\infty}(X_2,\Lambda^1(X_2))$ associated to it.

\begin{lemma}\label{covariant:derivatives:and:gluing:lem}
Let $t\in C^{\infty}(X_1\cup_f X_2,\Lambda^1(X_1\cup_f X_2))$, and let $x\in X_1\cup_f X_2$. Then:
$$(\nabla^{\cup}_t s)(x)=\left\{\begin{array}{cl}
(\tilde{\rho}_1^{\Lambda})^{-1}((\nabla^1_{t_1}s_1)(i_1^{-1}(x))) & \mbox{if }x\in i_1(X_1\setminus Y), \\
\left(\tilde{\rho}_1^{\Lambda}\oplus\tilde{\rho}_2^{\Lambda}\right)^{-1}\left((\nabla^1_{t_1}s_1)(\tilde{i}_1^{-1}(x))\oplus(\nabla^2_{t_2}s_2)(i_2^{-1}(x))\right) & \mbox{if }x\in i_2(f(Y)) \\
(\tilde{\rho}_2^{\Lambda})^{-1}((\nabla^2_{t_2}s_2)(i_2^{-1}(x))) & \mbox{if }x\in i_2(X_2\setminus f(Y)).
\end{array}\right.$$
\end{lemma}

\begin{proof}
The proof is by a straightforward calculation, and we limit ourselves to describing the more involved case, that of $x\in i_2(f(Y))$. It furthermore suffices to consider the case when 
$(\nabla^{\cup}s)(x)$ has form $\alpha(x)\otimes s(x)$, where $\alpha(x),s(x)\in\Lambda_x^1(X_1\cup_f X_2)$ (the rest is obtained by additivity). Then 
$$(\nabla^{\cup}_t s)(x)=g^{\Lambda}(x)(t(x),\alpha(x))s(x)=$$
$$=\frac12(g_1^{\Lambda}(\tilde{i}_1^{-1}(x))(t_1(\tilde{i}_1^{-1}(x)),\tilde{\rho}_1^{\Lambda}(\alpha(x)))\,\,+\,\,g_2^{\Lambda}(i_2^{-1}(x))(t_2(i_2^{-1}(x)),\tilde{\rho}_2^{\Lambda}(\alpha(x))))\cdot s(x).$$
The latter expression then amounts to
$$\frac12\left(\tilde{\rho}_1^{\Lambda}\oplus\tilde{\rho}_2^{\Lambda}\right)^{-1}(g_1^{\Lambda}(\tilde{i}_1^{-1}(x))(t_1(\tilde{i}_1^{-1}(x)),\tilde{\rho}_1^{\Lambda}(\alpha(x)))\cdot
(s_1(\tilde{i}_1^{-1}(x))\oplus s_2(i_2^{-1}(x)))\,\,+$$
$$+\,\, g_2^{\Lambda}(i_2^{-1}(x))(t_2(i_2^{-1}(x)),\tilde{\rho}_2^{\Lambda}(\alpha(x)))\cdot(s_1(\tilde{i}_1^{-1}(x))\oplus s_2(i_2^{-1}(x)))).$$ Collecting everything in the argument of 
$\left(\tilde{\rho}_1^{\Lambda}\oplus\tilde{\rho}_2^{\Lambda}\right)^{-1}$, the coefficient at $s_1(\tilde{i}_1^{-1}(x))$ is 
$$\frac12(g_1^{\Lambda}(\tilde{i}_1^{-1}(x))(t_1(\tilde{i}_1^{-1}(x)),\tilde{\rho}_1^{\Lambda}(\alpha(x)))+g_2^{\Lambda}(i_2^{-1}(x))(t_2(i_2^{-1}(x)),\tilde{\rho}_2^{\Lambda}(\alpha(x))))=
g_1^{\Lambda}(\tilde{i}_1^{-1}(x))(t_1(\tilde{i}_1^{-1}(x)),\tilde{\rho}_1^{\Lambda}(\alpha(x))),$$ where the equality is by the compatibility of the pseudo-metrics $g_1^{\Lambda}$ and $g_2^{\Lambda}$. 
Furthermore, $s_2(i_2^{-1}(x))$ has the same coefficient, which, for the same reason, is also equal to $g_2^{\Lambda}(i_2^{-1}(x))(t_2(i_2^{-1}(x)),\tilde{\rho}_2^{\Lambda}(\alpha(x)))$. It remains to observe 
that 
$$(\nabla^1_{t_1}s_1)(\tilde{i}_1^{-1}(x))=g_1^{\Lambda}(\tilde{i}_1^{-1}(x))(t_1(\tilde{i}_1^{-1}(x)),\tilde{\rho}_1^{\Lambda}(\alpha(x)))\cdot s_1(\tilde{i}_1^{-1}(x)),$$
$$(\nabla^2_{t_2}s_2)(i_2^{-1}(x))=g_2^{\Lambda}(i_2^{-1}(x))(t_2(i_2^{-1}(x)),\tilde{\rho}_2^{\Lambda}(\alpha(x)))\cdot s_2(i_2^{-1}(x)),$$
thus obtaining the desired expression, and so the claim.
\end{proof}

\subsubsection{The Lie bracket and gluing}

We now consider the behavior of the Lie bracket under gluing.

\begin{lemma}\label{section-lambda:acts:on:h:lem}
For every $h\in C^{\infty}(X_1\cup_f X_2,\matR)$, and for every $t\in C^{\infty}(X_1\cup_f X_2,\Lambda^1(X_1\cup_f X_2))$ we have:
$$t(h)(x)=\left\{\begin{array}{cl}
t_1(h_1)(i_1^{-1}(x)) & \mbox{if }x\in i_1(X_1\setminus Y), \\
\frac12 t_1(h_1)(\tilde{i}_1^{-1}(x))+\frac12 t_2(h_2)(i_2^{-1}(x)) & \mbox{if }x\in i_2(f(Y)), \\
t_2(h_2)(i_2^{-1}(x)) & \mbox{if }x\in i_2(X_2\setminus Y).
\end{array} \right.$$
\end{lemma}

\begin{proof}
For $x\in i_1(X_1\setminus Y)$ and $x\in i_2(X_2\setminus f(Y))$ the claim is straightforward, so let us consider $x\in i_2(f(Y))$. We have by definition and Proposition \ref{differential:under:gluing:prop} that
$$t(h)(x)=t(x)(dh(x))=t(x)\left((\tilde{\rho}_1^{\Lambda}\oplus\tilde{\rho}_2^{\Lambda})^{-1}(dh_1(\tilde{i}_1^{-1}(x))\oplus dh_2(i_2^{-1}(x)))\right),$$ where the meaning of the latter expression is 
\begin{flushleft}
$g^{\Lambda}(x)(t(x),(\tilde{\rho}_1^{\Lambda}\oplus\tilde{\rho}_2^{\Lambda})^{-1}(dh_1(\tilde{i}_1^{-1}(x))\oplus dh_2(i_2^{-1}(x))))=$
\end{flushleft}
\begin{flushright}
$=\frac12g_1^{\Lambda}(\tilde{i}_1^{-1}(x))(t_1(\tilde{i}_1^{-1}(x),dh_1(\tilde{i}_1^{-1}(x))))+\frac12 g_2^{\Lambda}(i_2^{-1}(x))(t_2(i_2^{-1}(x)),dh_2(i_2^{-1}(x))),$
\end{flushright}
which is equivalent to the desired expression, whence the claim.
\end{proof}

\begin{lemma}\label{lie:bracket:and:rho:lem}
For any $x\in i_2(f(Y))$ we have 
$$\tilde{\rho}_1^{\Lambda}([s,t](x))=[s_1,t_1](\tilde{i}_1^{-1}(x)) \,\,\mbox{ and }\,\,\tilde{\rho}_2^{\Lambda}([s,t](x))=[s_2,t_2](i_2^{-1}(x)).$$
\end{lemma}

\begin{proof}
We have by definition that $[s,t]=\Phi_{g^{\Lambda}}^{-1}([\Phi_{g^{\Lambda}}\circ s,\Phi_{g^{\Lambda}}\circ t])$. Let us consider $[\Phi_{g^{\Lambda}}\circ s,\Phi_{g^{\Lambda}}\circ t]$, which is determined 
by setting, for any arbitrary $r\in C^{\infty}(X_1\cup_f X_2,\Lambda^1(X_1\cup_f X_2))$, that 
$$[\Phi_{g^{\Lambda}}\circ s,\Phi_{g^{\Lambda}}\circ t](r)=(\Phi_{g^{\Lambda}}\circ s)((\Phi_{g^{\Lambda}}\circ t)(r))-(\Phi_{g^{\Lambda}}\circ t)((\Phi_{g^{\Lambda}}\circ s)(r)),$$ which, by definition still, is 
equal to
$$(\Phi_{g^{\Lambda}}\circ s)(g^{\Lambda}(t,r))-(\Phi_{g^{\Lambda}}\circ t)(g^{\Lambda}(s,r)).$$ At a point $x\in i_2(f(Y))$, by Lemma \ref{section-lambda:acts:on:h:lem}, this is
\begin{flushleft}
$\frac12(\Phi_{g_1^{\Lambda}}\circ s_1)((g^{\Lambda}(t,r))_1)(\tilde{i}_1^{-1}(x))+\frac12(\Phi_{g_2^{\Lambda}}\circ s_2)((g^{\Lambda}(t,r))_2)(i_2^{-1}(x))-$
\end{flushleft}
\begin{flushright}
$-\frac12(\Phi_{g_1^{\Lambda}}\circ t_1)((g^{\Lambda}(s,r))_1)(\tilde{i}_1^{-1}(x))-\frac12(\Phi_{g_2^{\Lambda}}\circ t_2)((g^{\Lambda}(s,r))_2)(i_2^{-1}(x)),$ 
\end{flushright}
which is in turn equal to 
\begin{flushleft}
$\frac12(\Phi_{g_1^{\Lambda}}\circ s_1)(g_1^{\Lambda}(t_1,r_1))(\tilde{i}_1^{-1}(x))-\frac12(\Phi_{g_1^{\Lambda}\circ t_1})(g_1^{\Lambda}(s_1,r_1))(\tilde{i}_1^{-1}(x))+$
\end{flushleft}
\begin{flushright}
$+\frac12(\Phi_{g_2^{\Lambda}}\circ s_2)(g_2^{\Lambda}(t_2,r_2))(i_2^{-1}(x))-\frac12(\Phi_{g_2^{\Lambda}\circ t_2})(g_2^{\Lambda}(s_2,r_2))(i_2^{-1}(x)).$
\end{flushright}
The final conclusion is therefore that 
$$[\Phi_{g^{\Lambda}}\circ s,\Phi_{g^{\Lambda}}\circ t](r)(x)=\frac12([s_1,t_1](r_1)(\tilde{i}_1^{-1}(x))+[s_2,t_2](r_2)(i_2^{-1}(x))),$$
which yields 
$$(\tilde{\rho}_1^{\Lambda}\oplus\tilde{\rho}_2^{\Lambda})([s,t])(x)=([s_1,t_1](r_1)(\tilde{i_1^{-1}(x)})+[s_2,t_2](r_2)(i_2^{-1}(x))),$$
as wanted.
\end{proof}

We thus immediately obtain the following statement.

\begin{prop}\label{lie:bracket:and:gluing:prop}
Let $s,t,\in C^{\infty}(X_1\cup_f X_2,\Lambda^1(X_1\cup_f X_2))$, let $s_1,s_2,t_1,t_2$ be as in Section \ref{splitting:sections:sect}, and let $x\in X_1\cup_f X_2$. Then
\emph{$$[s,t](x)=\left\{\begin{array}{cl}
(\tilde{\rho}_1^{\Lambda})^{-1}([s_1,t_1](i_1^{-1}(x))) & \mbox{if }x\in i_1(X_1\setminus Y), \\ 
(\tilde{\rho}_1^{\Lambda}\oplus\tilde{\rho}_2^{\Lambda})^{-1}(\mbox{[}s_1,t_1\mbox{]}(\tilde{i}_1^{-1}(x))\oplus\mbox{[}s_2,t_2\mbox{]}(i_2^{-1}(x))) & \mbox{if }x\in i_2(f(Y)), \\ 
(\tilde{\rho}_2^{\Lambda})^{-1}(\mbox{[}s_2,t_2\mbox{]}(i_2^{-1}(x))) & \mbox{if }x\in i_2(X_2\setminus f(Y)).
\end{array}\right.$$}
\end{prop}

\subsubsection{The behavior of the torsion tensor}

It now remains to consider the corresponding torsion tensors. 

\begin{lemma}\label{torsion:tensor:over:factors:lem}
Let $s,t\in C^{\infty}(X_1\cup_f X_2,\Lambda^1(X_1\cup_f X_2))$, let $s_1,s_2,t_1,t_2$ be as previously defined, and let $x\in X_1\cup_f X_2$. Then: 
$$T^{\nabla^{\cup}}(s,t)(x)=\left\{\begin{array}{cl}
(\tilde{\rho}_1^{\Lambda})^{-1}\left(T^{\nabla^1}(s_1,t_1)(i_1^{-1}(x))\right) & \mbox{if }x\in i_1(X_1\setminus Y), \\
\frac12(\tilde{\rho}_1^{\Lambda}\oplus\tilde{\rho}_2^{\Lambda})^{-1}\left(T^{\nabla^1}(s_1,t_1)(\tilde{i}_1^{-1}(x)))\oplus T^{\nabla^2}(s_2,t_2)(i_2^{-1}(x))\right) & \mbox{if }x\in i_2(f(Y)), \\
(\tilde{\rho}_2^{\Lambda})^{-1}\left(T^{\nabla^2}(s_2,t_2)(i_2^{-1}(x))\right) & \mbox{if }x\in i_2(X_2\setminus f(Y)).
\end{array}\right.$$
\end{lemma}

\begin{proof}
This is an immediate consequence of Lemma \ref{covariant:derivatives:and:gluing:lem} and Proposition \ref{lie:bracket:and:gluing:prop}.
\end{proof}

We can now immediately conclude the following:

\begin{thm}\label{gluing:symmetric:connections:yields:symmetric:connection:thm}
Let $X_1$ and $X_2$ be two diffeological spaces, and let $f:X_1\supseteq Y\to X_2$ be a diffeomorphism such that $\calD_1^{\Omega}=\calD_2^{\Omega}$. Let $\nabla^1$
be a connection on $\Lambda^1(X_1)$, let $\nabla^2$ be a connection on $\Lambda^1(X_2)$, and assume that $\nabla^1$ and $\nabla^2$ are compatible with the gluing along $f$, and
that each of them is symmetric. Then the induced connection $\nabla^{\cup}$ on $\Lambda^1(X_1\cup_f X_2)$ is symmetric as well.
\end{thm}

\begin{proof}
It suffices to observe that if $T^{\nabla^1}$ and $T^{\nabla^2}$ are both trivial, then by Lemma \ref{torsion:tensor:over:factors:lem} the torsion tensor $T^{\nabla^{\cup}}$ of the induced connection is trivial 
as well.
\end{proof}

\begin{cor}
If $\nabla^1$ and $\nabla^2$ are the Levi-Civita connections then $\nabla^{\cup}$ is the Levi-Civita connection as well.
\end{cor}

\vspace{1cm}

\noindent University of Pisa \\
Department of Mathematics \\
Via F. Buonarroti 1C\\
56127 PISA -- Italy\\
\ \\
ekaterina.pervova@unipi.it\\

\end{document}